\documentclass[12pt]{amsart}
\usepackage{amscd,amsmath,amsthm,amssymb}
\usepackage{pstcol,pst-plot,pst-3d}%\usepackage[T1]{fontenc}
\usepackage{lmodern,pst-node}%\usepackage{geometric}
\usepackage{multicol}
\usepackage{epic,eepic}
\usepackage{amsfonts,amssymb,amscd,amsmath,enumerate,verbatim}
\psset{unit=0.7cm,linewidth=0.8pt,arrowsize=2.5pt 4}
\newpsstyle{fatline}{linewidth=1.5pt}
\newpsstyle{fyp}{fillstyle=solid,fillcolor=verylight}
\definecolor{verylight}{gray}{0.97}
\definecolor{light}{gray}{0.9}
\definecolor{medium}{gray}{0.85}
\definecolor{dark}{gray}{0.6}
 \def\opn#1#2{\def#1{\operatorname{#2}}} % to make operators
 \opn\chara{char} \opn\length{\ell} \opn\pd{pd} \opn\rk{rk}
 \opn\projdim{proj\,dim} \opn\injdim{inj\,dim} \opn\rank{rank}
 \opn\depth{depth} \opn\grade{grade} \opn\height{height}
 \opn\embdim{emb\,dim} \opn\codim{codim}
 
 \opn\Tr{Tr} \opn\bigrank{big\,rank}
 \opn\superheight{superheight}\opn\lcm{lcm}
 \opn\trdeg{tr\,deg}%\emph{
 \opn\reg{reg} \opn\lreg{lreg} \opn\ini{in} \opn\lpd{lpd}
 \opn\size{size} \opn\sdepth{sdepth}
 \opn\link{link}\opn\fdepth{fdepth}\opn\lex{lex}
 \opn\div{div} \opn\Div{Div} \opn\cl{cl} \opn\Cl{Cl}
 \opn\Spec{Spec} \opn\Supp{Supp} \opn\supp{supp} \opn\Sing{Sing}
 \opn\Ass{Ass} \opn\Min{Min}\opn\Mon{Mon}
 \opn\Ann{Ann} \opn\Rad{Rad} \opn\Soc{Soc}
 \opn\Im{Im} \opn\Ker{Ker} \opn\Coker{Coker} \opn\Am{Am}
 \opn\Hom{Hom} \opn\Tor{Tor} \opn\Ext{Ext} \opn\End{End}
 \opn\Aut{Aut} \opn\id{id}
 
 \opn\nat{nat}
 \opn\pff{pf}%   \pf exists already
 \opn\Pf{Pf} \opn\GL{GL} \opn\SL{SL} \opn\mod{mod} \opn\ord{ord}
 \opn\Gin{Gin} \opn\Hilb{Hilb}\opn\sort{sort}
 \opn\aff{aff} \opn
 \con{conv} \opn\relint{relint} \opn\st{st}
 \opn\lk{lk} \opn\cn{cn} \opn\core{core} \opn\vol{vol}  \opn\inp{inp} \opn\nilpot{nilpot}
 \opn\link{link} \opn\star{star}\opn\lex{lex}\opn\set{set}
 \opn\gr{gr}
 
 \def\pot#1#2{#1[\kern-0.28ex[#2]\kern-0.28ex]}

 \opn\dirlim{\underrightarrow{\lim}}
 \opn\inivlim{\underleftarrow{\lim}}
 \let\union=\cup
 \let\sect=\cap

 \let\to=\rightarrow
 
 \def\Implies{\ifmmode\Longrightarrow \else
         \unskip${}\Longrightarrow{}$\ignorespaces\fi}
 \def\implies{\ifmmode\Rightarrow \else
         \unskip${}\Rightarrow{}$\ignorespaces\fi}
 \def\iff{\ifmmode\Longleftrightarrow \else
         \unskip${}\Longleftrightarrow{}$\ignorespaces\fi}

 \let\:=\colon
 \newtheorem{Theorem}{Theorem}[section]
 \newtheorem{Lemma}[Theorem]{Lemma}
 \newtheorem{Corollary}[Theorem]{Corollary}
 \newtheorem{Proposition}[Theorem]{Proposition}
 \newtheorem{Remark}[Theorem]{Remark}

 \newtheorem{Examples}[Theorem]{Examples}
 \newtheorem{Definition}[Theorem]{Definition}

 \let\epsilon\varepsilon
 \let\kappa=\varkappa
 \def\qed{\ifhmode\textqed\fi
       \ifmmode\ifinner\quad\qedsymbol\else\dispqed\fi\fi}
 \def\textqed{\unskip\nobreak\penalty50
        \hskip2em\hbox{}\nobreak\hfil\qedsymbol
        \parfillskip=0pt \finalhyphendemerits=0}
 \def\dispqed{\rlap{\qquad\qedsymbol}}
 
 \opn\dis{dis}
 \def\pnt{{\raise0.5mm\hbox{\large\bf.}}}
 
 \opn\Lex{Lex}

\begin{document}

 \title {The binomial edge ideal of a pair of graphs}

 \author {Viviana Ene, J\"urgen Herzog, Takayuki Hibi and Ayesha Asloob Qureshi}

\address{Viviana Ene, Faculty of Mathematics and Computer Science, Ovidius University, Bd.\ Mamaia 124,
 900527 Constanta, Romania} \email{vivian@univ-ovidius.ro}

\address{J\"urgen Herzog, Fachbereich Mathematik, Universit\"at Duisburg-Essen, Campus Essen, 45117
Essen, Germany} \email{juergen.herzog@uni-essen.de}

\address{Takayuki Hibi, Department of Pure and Applied Mathematics, Graduate School of Information Science and Technology,
Osaka University, Toyonaka, Osaka 560-0043, Japan}
\email{hibi@math.sci.osaka-u.ac.jp}

\address{Ayesha Asloob Qureshi, Abdus Salam School of Mathematical Sciences, GC University,
Lahore. 68-B, New Muslim Town, Lahore 54600, Pakistan} \email{ayesqi@gmail.com}

 \begin{abstract}
We introduce a class of ideals generated by a set of $2$-minors of $m\times n$-matrix of indeterminates indexed by a pair of graphs. This class of ideals is a natural common generalization of binomial edge ideals and ideals generated by adjacent minors. We determine the minimal prime ideals of such ideals and give a lower bound for their degree of nilpotency. In some special cases we compute their Gr\"obner basis  and characterize unmixedness and Cohen--Macaulayness.
 \end{abstract}

\thanks{The first author was supported  by the JSPS Invitation Fellowship Programs for Research in Japan. Part of this paper was done during the visit of second author to Abdus Salam School of Mathematical Sciences. The fourth author contributed to this paper during her post-doctoral fellowship at Abdus Salam School of Mathematical Sciences.}

 \maketitle

 \section*{Introduction}
The study of ideals generated by minors of a generic matrix, mostly motivated by geometric questions,  has a long tradition, see the fundamental papers \cite{St90}, \cite{Ho} and the survey \cite{BC}. Classically these are ideals generated by all minors of a given size. More recently, research has focussed on ideals generated by arbitrary  sets of  minors of a generic matrix. Perhaps the first paper in this direction is that of Andrade \cite{A} from the 1981 in which regular sequences of minors are considered. In the last years, due to techniques used in algebraic statistic, it proves necessary  to study certain classes of binomial and determinantal ideals.  This includes ideals generated by adjacent minors, as introduced by Diaconis,  Eisenbud and Sturmfels \cite{DES} and further studied in  \cite{HS}, \cite{HSS} and \cite{HH}, as well the binomial edge ideals,  first considered in \cite{HHHKR} and recently generalized by \cite{RA}. The algebraic properties of this class of ideals  are widely  open, though several partial results are known, see for example \cite{EHH}.  From an algebraic point of view we are interested in the following questions: what are the associated primes of these ideals, and in particular their minimal primes, what is their Gr\"obner basis, when are these ideals reduced or prime, when are they Cohen-Macaulay or Gorenstein?

In this paper we introduce binomial edge ideals $J_{G_1,G_2}$  attached to a pair $(G_1,G_2)$  of finite graphs. This class of ideals generalizes the versions of binomial edge ideals,  considered in \cite{HHHKR} and \cite{RA}, but also includes ideals generated by adjacent minors which turn out to be the ideals attached to a pair of line graphs.

In Section~1 we study the Gr\"obner basis of these ideals. A  general description of these Gr\"obner bases seems to be extremely difficult. However in Theorem~\ref{quadratic} we succeed to classify those pairs of graphs for which $J_{G_1,G_2}$ has a quadratic Gr\"obner basis. Unlike to the classical binomial edge ideals, the binomial edge ideals attached to a pair of graphs are never radical, unless $G_1$ or $G_2$ is complete, see Theorem~\ref{radical}.

In Theorem~\ref{minprimes} of  Section~2 we describe quite explicitly the minimal prime ideals of $J_{G_1,G_2}$. They are essentially determined by the  so-called admissible sets of variables which are determined by data of the two graphs. The results obtained in Section~2 are applied in Section~3 to give a detailed description of all minimal prime ideals in the case that $G_1$ is a line graph of length 2 and $G_2$ is an arbitrary graph. The information on the minimal prime ideals is also used in the following Section~4 where the unmixed binomial edge ideals of pairs of graphs are characterized in Proposition~\ref{unmixed}. The condition for being unmixed is,  that one of the graphs is complete and the other graph satisfies certain numerical conditions related to its sets having the cut point property. In the case that one graph is complete and the other one is a cycle we fully classify in Proposition~\ref{unmixed cycle} the unmixed  binomial edge ideals. Though the  conditions  guaranteeing that the binomial edge ideals of a pair of graphs is unmixed are already pretty restrictive, the more they are restrictive for them to be Cohen--Macaulay. Under the assumption that $G_1$ is complete and $G_2$ is closed in the sense of \cite{HHHKR} and  $|V(G_2)|\geq |V(G_1)|\geq 3$, the unmixedness $J_{G_1,G_2}$ is characterized and the depth of $S/J_{G_1,G_2}$ is computed, see Theorem~\ref{Osaka}.  It follows that,  under the assumptions of the theorem, $J_{G_1,G_2}$ is Cohen--Macaulay only if both graphs are complete.

 For an ideal $I$ with radical $\sqrt{I}$,  the least number $k$   with the property that $(\sqrt{I})^k \subset I$ is called  index of nilpotency of $I$, and denoted $\nilpot(I)$. It is clear that $\nilpot(I)=1$ if and only if $I$ is a radical ideal. Thus, as noticed before,  $\nilpot(J_{G_1,G_2})=1$ if and only if $G_1$ or $G_2$ is complete. In the last section of this paper we give in Theorem~\ref{nilpot} a  lower bound for index of nilpotency of $J_{G_1,G_2}$ in terms of data of the graphs $G_1$ and $G_2$. Applying this result to an $m\times n$-matrix of adjacent minors one obtains that this lower bound is approximately $mn/16$.

 \section{Binomial edge ideals of pairs of graphs and their  Gr\"obner basis}
 Let $G_1$ be a graph on the vertex set $[m]$ and  $G_2$ a graph on the vertex set $[n]$. We fix a field $K$, let $X=(x_{ij})$ be an $(m\times n)$-matrix of indeterminates,  and denote by $K[X]$ the polynomial ring in the variables $x_{ij}$, $i=1,\ldots,m$ and $j=1,\ldots,n$.

 Let $e=\{i,j\}$ for some  $1 \leq i<j \leq m$ and $f=\{k,l\}$ for some $1 \leq k<l \leq n$. To the pair $(e,f)$ we assign the following $2$-minor of $X$:
 \[
 p_{e,f}=[i,j|k, l]=x_{ik}x_{jl}-x_{il}x_{jk}.
 \]
 The ideal
 \[
 J_{G_1,G_2}=(p_{e,f}\:\; e\in E(G_1), f\in E(G_2))
 \]
 is called the {\em binomial edge ideal} of the pair $(G_1,G_2)$.

 \begin{Examples}
 {\em
 (a) If $G_1$ and $G_2$ are complete graphs, then $J_{G_1,G_2}= I_2(X)$, the ideal of all $2$-minors of $X$.

 (b) If $G_1$ is  the graph consisting of exactly one edge, then $J_{G_1,G_2}$  is the binomial edge ideal $J_{G_2}$ introduced in \cite{HHHKR}.

 (c) If $G_1$ is a complete graph, then  $J_{G_1,G_2}$ is the generalized  binomial edge ideal attached to $G_2$, as considered in \cite{RA}.

 (d) If $G_1$ and $G_2$ are line graphs, then $J_{G_1,G_2}$  is the ideal of adjacent $2$-minors of the  matrix $X$, studied in \cite{DES}, \cite{HS} and \cite{HSS}.
 }
 \end{Examples}

 \begin{Theorem}
 \label{radical}
 Let $J_{G_1,G_2}$ be the binomial edge ideal of the pair of graphs $(G_1,G_2)$. Then the following conditions are equivalent:
 \begin{enumerate}
 \item[{\em (a)}] $J_{G_1,G_2}$ is a radical ideal, that is, $J_{G_1,G_2}=\sqrt{J_{G_1,G_2}}$.
 \item[{\em (b)}] $J_{G_1,G_2}$ has a squarefree Gr\"obner basis with respect to the lexicographic order induced by
 \[
 x_{11}> x_{12}> \cdots > x_{1n}> x_{21}> x_{22}> \cdots > x_{mn}.
 \]
 \item[{\em (c)}] Either $G_1$ or $G_2$ is a  complete graph.
 \end{enumerate}
 \end{Theorem}

 \begin{proof} The implication (c) \implies (b) is shown in \cite{RA}, and (b)\implies (a) is a general fact, see for example proof of \cite[Corollary 2.2]{HHHKR}. Thus it remains to be shown that (a) implies (c).  Suppose that neither $G_1$ nor $G_2$ is a complete graph. Then there exist subsets $T_1\subset [m]$ and $T_2\subset [n]$ such that the restrictions $L_1= (G_1)_{T_1}$ and $L_2=(G_2)_{T_2}$ are line graphs,  each of them with two edges, say, $E(L_1)= \{\{i,j\},\{j,k\}\}$ and   $E(L_2)= \{\{r,s\}, \{s,t\}\}$. Then the element
 \begin{eqnarray}
 \label{twolines}
 f_{L_1, L_2}=x_{it}x_{jr}x_{ks}-x_{ir}x_{js}x_{kt}
 \end{eqnarray}
 does not belong to the ideal
 \[
 I=(p_{e,f}\:\; e\in E(L_1), f\in E(L_2)),
 \]
and hence $ f_{L_1, L_2} \notin J_{G_1,G_2}$ because $I$ is obtained from $J_{G_1, G_2}$ by substituting all the variables by 0 which do not appear among the generators of $I$. On the other hand $ f^2_{L_1, L_2} \in I$, and hence
$ f^2_{L_1, L_2}\in J_{G_1,G_2}$. This shows that $J_{G_1,G_2}$ is not a radical ideal.
\end{proof}

In \cite{HHHKR} the concept of a closed graph is introduced. Recall that a graph $G$ on the vertex set $[n]$ is called {\em closed} if for all edges   $\{i,j\}$ and $\{k,l\}$ of $G$  with $i<j$ and $k<l$ one has  $\{j,l\}\in E(G)$ if $i=k$, and $\{i,k\}\in E(G)$ if $j=l$.

\medskip
The next result shows that only in exceptional cases the binomial generators of $J_{G_1,G_2}$ form a Gr\"obner basis of $J_{G_1,G_2}$.

\begin{Theorem}
\label{quadratic}
Let $J_{G_1,G_2}$ be the binomial edge ideal of the pair of graphs $(G_1, G_2)$. Then the following conditions are equivalent:
\begin{enumerate}
\item[{\em (a)}] $J_{G_1,G_2}$ has a quadratic Gr\"obner basis with respect to the monomial order introduced in Theorem~\ref{radical}.
\item[{\em (b)}] $G_1$ is complete and $G_2$ is closed, or vice versa.
\end{enumerate}
\end{Theorem}

\begin{proof}
(a) \implies (b): Since the quadratic Gr\"obner basis of $J_{G_1,G_2}$ consists of binomials with squarefree terms, it follows that $J_{G_1,G_2}$ is a radical ideal. Therefore, by Theorem~\ref{radical}, one of the graphs must be complete. Let us assume that $G_1$ is complete, and show that $G_2$ is closed. Let $\{i,j\}$ be an edge of $G_1$ and $\{k,l\}$, $\{k,q\}$ be two edges of $G_2$ with $k<l$ and $k<q$. Then the $S$-polynomial $S(f,g)$ for $f=x_{ik} x_{jl} - x_{jk} x_{il}$ and $g=x_{ik} x_{jq} - x_{iq} x_{jk}$ has the initial monomial $x_{jq} x_{jk} x_{il}$, and, since $J_{G_1,G_2}$ has quadratic Gr\"obner basis, we must have the edge $\{l,q\}$ in $G_2$.

(b) \implies (a): Let $p_{e, f}, p_{e', f'} \in J_{G_1, G_2}$. We show that $S(p_{e, f}, p_{e', f'})$ reduces to zero. If the initial terms of $ p_{e, f}, p_{e', f'}$ are coprime, then there is nothing to prove. Let $e=\{i,j\}, f=\{k,l\}, e'=\{i',j'\}, f'=\{k',l'\}$ with $i < j$, $k < l$, $i'<j'$, $k'< l'$. The initial terms of $p_{e, f}, p_{e', f'} $ have a common factor if and only if (1) $i=i'$ and $k=k'$, (2) $j=j'$ and $l=l'$, or (3) $j=i'$ and $l=k'$. It is straightforward to verify in all the three case that $S(p_{e, f}, p_{e', f'})$ reduces to zero. For example, in case (1), if $j < j'$ and $l<l'$ then  completeness of $G_1$ gives $g=\{j,j'\} \in E(G_1)$ and closedness of $G_2$ implies that $h=\{l,l'\} \in E(G_2)$. Then $S(p_{e, f}, p_{e', f'})$ reduces to 0 with respect to $p_{e',h}$ and $p_{g,f}$.
\end{proof}

%Similarly, in case (2), completeness of $G_1$ gives $g'=\{i,i\} \in E(G_1)$ and closedness of $G_2$ implies that $h'=\{k,k'\} \in E(G_2)$. Then $S(p_{g',f}, p_{e', h'})$ reduces to 0 with respect to $p_{e',h}$ and $p_{g,f}$.

%In case(3), completeness of $G_1$ gives $a=\{i,j'\} \in E(G_1)$, and $S(p_{g',f}, p_{e', h'})$ reduces to 0 with respect to $p_{a,f}$, $p_{a',f'}$, $p_{e,f'}$ and $p_{e',f}$.
%In this case, the  only  $S$-polynomial which may give a non-zero remainder of degree 3 is $S(ae-bd, ei-fh)$. But we have
%\[ S(ae-bd,ei-fh) = f(ah-bg)-d(bi-ch)+g(bf-ec)-c(dh-ge) \]

\section{The minimal prime ideals}

Let $J_{G_1,G_2} \subset K[X]$ be the binomial edge ideal of the pair of graphs $(G_1,G_2)$. Our aim is to describe the minimal prime ideals of $J_{G_1,G_2}$. This will be done in several steps. Throughout this section we will assume that $G_1$ and $G_2$ are both connected.

\begin{Lemma}
\label{novariables}
The ideal of all $2$-minors of $I_2(X)$ is a minimal prime ideal of $J_{G_1,G_2}$, and if $P$ is a minimal prime ideal of $J_{G_1,G_2}$ containing no variable, then $P= I_2(X)$.
\end{Lemma}

\begin{proof}
Let $x=\prod_{i=1,\ldots,m\atop j=1,\ldots n}x_{ij}$. We claim that  $J_{G_1,G_2} \: x^\infty=I_2(X)$. This  will then  imply the assertions of the lemma. Because if $P$ is a minimal  ideal of $J_{G_1,G_2}$  not containing a variable,  then $ J_{G_1,G_2}\subset I_2(X)= J_{G_1,G_2}\: x^\infty\subset P\: x^\infty=P$, and hence $P$ is equal to  $I_2(X)$.

In order to prove the claim, let $\delta =[i,j|k,l]$ be an arbitrary $2$-minor of $X$. We will show that $\delta\in  J_{G_1,G_2} \: x^\infty$. Assuming this we conclude that $I_2(X)\ : x^\infty=  J_{G_1,G_2} \: x^\infty$. However since $I_2(X)$ is a prime ideal, we then have $I_2(X)\: x^\infty=I_2(X)$,  and the claim  is proven.

To see that $\delta \in  J_{G_1,G_2}\: x^\infty$, we observe that there is a path $P_1$ in $G_1$ from $i$ to $j$, that is, a sequence $i=i_0,i_1,\ldots, i_{r-1},i_r=j$ such that
$\{i_s,i_{s+1}\}\in E(G_1)$ for $s=0,\ldots,r-1$.  The number $r$ is called the length of the path. Similarly there exists a path $P_2\: k=k_0,k_1,\ldots, k_{t-1},k_t=l$ in $G_2$ from $k$ to $l$. We will show by induction on $r+t$, that  $\delta \in J_{G_1,G_2}\: x^\infty$. Notice that $r+t\geq 2$. If $r+t=2$, then $\delta \in J_{G_1,G_2}$, and the assertion is trivial. Suppose now that $r+t>2$. We may assume that $r>1$. By applying the induction hypothesis, we have that
$\delta_1=[i,i_{r-1}|k,l]$ and $\delta_2=[i_{r-1},j|k,l]$ belong to $J_{G_1,G_2} \: x^\infty$. Since $x_{i_{r-1}k}\delta=x_{ik}\delta_2+x_{jk}\delta_1$, it follows that  $\delta \in J_{G_1,G_2}\: x^\infty$, as desired.
\end{proof}

\begin{Corollary}
\label{obvious}  $J_{G_1,G_2}$ is a prime ideal if and only if $G_1$ and $G_2$ are complete graphs.
\end{Corollary}

Next we are going to study minimal prime ideals of $J_{G_1,G_2}$ which contain variables. In this context the following definition turns out to be useful.

\begin{Definition}
{\em A subset $W \subset [m]\times [n]$ is called {\em admissible} with respect to $(G_1,G_2)$ if it satisfies the following property: whenever  $(i,j)\in e\times f\sect W$ for some $e\in E(G_1)$ and some $f\in E(G_2)$, then $\{i\}\times f \subset W$ or $e \times\{j\}\subset W$.}
\end{Definition}

Obviously, the empty set  and the set $[m]\times [n]$ are admissible.

If $e=\{i,j\}\in E(G_1)$ and $f=\{k,l\}\in E(G_2)$, then the sets $\{i\}\times f$, $\{j\}\times f$, $e\times\{k\}$ and $e\times \{l\}$ are called the {\em edges} of $e\times f$. An admissible set $W$ with respect to $(G_1,G_2)$ is characterized by the property that if $W\sect (e\times f) \neq \emptyset$, then one  of  the edges of $e\times f$ is contained in $W$.

The significance of admissible sets  for the study of the minimal prime ideals of $J_{G_1,G_2}$ becomes apparent  by the next result.

\begin{Lemma}
\label{admissible}
Let $P$ be a prime ideal containing $J_{G_1,G_2}$, and let $W=\{(i,j)\:\; x_{ij}\in P\}$. Then $W$ is an admissible set.
\end{Lemma}

\begin{proof}
Let $(i,j) \in W$. Then $x_{ij} \in P$.  Assume $(i,j) \in e \times f$, with $e=\{i,k\}$ and $f=\{j,l\}$. Then $x_{ij} x_{kl} - x_{il} x_{kj} \in J_{G_1,G_2} \subset P$. This implies $x_{il} x_{kj} \in P$. Since $P$ is prime, we either have $x_{il} \in P$, or $x_{kj} \in P$. If $x_{il} \in P$, then $\{i\}\times f \subset W$. Otherwise, we have $x_{kj} \in P$, and then $e \times\{j\}\subset W$.
\end{proof}

We call a subset $E\subset E(G_1)\times E(G_2)$  {\em connected}, if for all $e\times f$ and $e'\times f'$  in $E$ there exist  $e_i\times f_i\in E$, $i=1,\ldots,r$ such that $e_i\times f_i =e_1\times f_1$, $e'\times f' =e_r\times f_r$ and $(e_i\times f_i) \sect (e_{i+1}\times f_{i+1})\neq  \emptyset$ for $i=1,\ldots,r-1$.

An arbitrary subset $E\subset E(G_1)\times E(G_2)$   can be uniquely written as a disjoint union of connected subsets of  $E(G_1)\times E(G_2)$, called the {\em connected components} of $E$.

\begin{Lemma}
\label{connected}
Let $W \subset [m]\times [n]$ be an  admissible set with respect to $(G_1,G_2)$. Then the connected components of
\[
W^c=\{e\times f \:\; e\in E(G_1),\; f\in E(G_2),\; W\sect (e\times f)=\emptyset\}.
\]
are of the form $E_1\times E_2$ where $E_1\subset E(G_1)$ and $E_2\subset E(G_2)$.
\end{Lemma}

\begin{proof}
Let $e\times f$ and $e'\times f'$ belong to the same connected component $C$ of $W^c$. Then there exist  $e_i\times f_i\in C$, $i=1,\ldots,r$ such that $e\times f =e_1\times f_1$, $e'\times f' =e_r\times f_r$ and $(e_i\times f_i )\sect (e_{i+1}\times f_{i+1}) \neq  \emptyset$ for $i=1,\ldots,r-1$.

We have to show that $e\times f'\in C$ and $e'\times f\in C$. We show this by induction on $r$. The assertion is trivial, if $r=1$. Now let $r>1$, and assume the assertion is already shown for $r-1$. Then, since $e_2 \times f_2$ is connected in $C$ to $e_r \times f_r$ by a chain of length $r-1$, the induction hypothesis implies that  $e_2 \times f_r$  belongs to $C$. Similarly, since  $e_{r-1} \times f_{r-1}$ is connected in $C$ to $e_1\times f_1$  by a chain of length $r-1$, we have $e_1 \times f_{r-1}$ in $C$. %Notice that $(e_2 \times f_r) \cap (e_1 \times f_{r-1}) \neq \emptyset$ because $e_1 \cap e_2 \neq \emptyset$ and $f_r \cap f_{r-1} \neq \emptyset$.
Suppose that $e_1 \times f_r \notin C$. Then $e_1\neq e_2$ and $f_{r-1}\neq f_r$, and moreover  $(e_1 \times f_r) \cap W \neq \emptyset$, say $(i,j) \in (e_1 \times f_r) \cap W$. Since $W$ is admissible, it follows that either $\{i\} \times f_r \in W$ or $e_1 \times \{j\} \in W$. This implies $(e_2 \times f_r )\cap W \neq \emptyset$ or $e_1 \times f_{r-1} \neq \emptyset$. It follows that $e_2 \times f_r \notin C$ or $e_1 \times f_{r-1} \notin C$, a contradiction. Hence we conclude that $e \times f' = e_1 \times f_r \in C$. Similarly, one can show that $e' \times f \in C$.
\end{proof}

Let $W$ be an admissible subset of $G_1 \times G_2$, and let $C_1, \ldots, C_r$ be the connected components of $W^c$ in the graph $G_1 \times G_2$. The set of edges of $G_1 \times G_2$ is defined to be the set
$\{\{\{i,j\}, \{k,l\}\}\:\; \{i,j\} \in E(G_1), \{k,l\} \in E(G_2)\}$.  By Lemma~\ref{connected}, there exists subgraphs $G_{1i} \subset G_1$ and $G_{2i} \subset G_2$ such that $C_i = E(G_{1i}) \times E(G_{2i})$. Since all $C_i$ are connected, it follows that the graphs  $G_{1i}$ and $G_{2i}$ are connected, and
\[
W^c = \bigsqcup_{i} E(G_{1i}) \times E(G_{2i})),
\]
where $\bigsqcup$ denote the disjoint union.

For a graph $G$, we define $\hat{G}$ to be the complete graph on the vertex set $V(G)$. By using this notation, we define
\[
\widehat{W}^c = \bigsqcup_{i} E(\widehat{G}_{1i}) \times E(\widehat{G}_{2i}))
\]

Obviously, the ideal
\[
P_W=(\{x_{ij} : (i,j) \in W \}, Q_W) \quad \text{with}\quad Q_W=( p_{e,f} \; \: \; e \times f \in \widehat{W}^c))
\]
is a prime ideal.

\begin{Proposition} \label{contained}
Let $V$ and $W$ be two admissible sets with respect to $(G_1, G_2)$. Then the following conditions are equivalent:
\begin{enumerate}
\item[{\em (a)}] $P_V \subsetneq P_W$.
\item[{\em (b)}] $V \subsetneq W$ and for all $e \times f \subset \widehat{V}^c \setminus \widehat{W}^c$ an edge of $e \times f$ belongs to $W$.
\end{enumerate}
\end{Proposition}
\begin{proof}
 (a) $\Rightarrow $ (b): Let $(i,j) \in V$. Then $x_{ij} \in P_V \subset P_W$. This implies that $(i,j) \in W$. Therefore, $V \subset W$. The inclusion must be proper, otherwise $P_V = P_W$. Assume that $e \times f \subset \widehat{V}^c \setminus \widehat{W}^c$. Then $p_{e,f} \in Q_V \setminus Q_W$. This implies $p_{e,f} \in P_W \setminus Q_W$. Therefore, some corner of $e \times f$ belongs to $W$. Since $P_W$ is a prime ideal, an edge of $e \times f$ belongs to $W$.

 (b) $\Rightarrow $ (a): The inclusion $V \subsetneq W$  implies that $\{x_{ij} : (i,j) \in V \}\subsetneq \{x_{ij} : (i,j) \in W \}$. If there  exist $p_{e,f} \in Q_V \setminus Q_W$, then   $e \times f \subset \widehat{V}^c \setminus \widehat{W}^c$. By our assumption, this implies that an edge of $e \times f$ belongs to $W$. Therefore $p_{e,f} \in (x_{ij}, \{i,j\} \in W)$.  This shows that  $P_V \subsetneq P_W$.
\end{proof}

\begin{Theorem} \label{minprimes}
{\em (a)} Let $P$ be a minimal  prime ideal of the binomial edge ideal $J_{G_1,G_2}$ of the pair $(G_1 , G_2)$. Then there exists an  admissible set $W \subset G_1 \times G_2$ such that $P=P_W$.

{\em (b)} Let $W \subset G_1 \times G_2$  be an admissible set. Then $P_W$ is a minimal prime ideal of $J_{G_1,G_2}$, if and only if for any admissible set $V \subset G_1 \times G_2$  properly contained in $W$ there exists $e\times f\in \widehat{V}^c\setminus \widehat{W}^c$ such that no edge of $e\times f$ belongs to $W$.
\end{Theorem}

\begin{proof}
(a) Let   $W=\{(i,j)\:\; x_{ij}\in P\}$. Then  $(\{x_{ij} : (i,j) \in W \}, J_{G_1,G_2})\subset P$,  and
$(\{x_{ij} : (i,j) \in W \}, J_{G_1,G_2})=(\{x_{ij} : (i,j) \in W \}, Q)$, where $Q$ is generated by all minors $p_{e,f}$
such that  $W$ does not contain an edge of  $e\times f$. Hence,  since $W$ is admissible, as is shown in Lemma~\ref{admissible}, it follows that   $Q=(\{p_{e,f}\:\,  e\times f\in \widehat{W}^c\})$. Now we apply Lemma~\ref{connected} and conclude that
$Q=\sum_{i=1}^r J_{G_{1i}, G_{2i}}$, where $C_1,\ldots, C_r$  are the connected components $\widehat{W}^c$ and $C_i=E(G_{1i})\times E(G_{2i})$, as described in Lemma~\ref{connected} and the comments following it.

Thus our discussion so far shows that $P$ is a minimal prime ideal of
\[
Q=(\{x_{ij} : (i,j) \in W \}, \sum_{i=1}^r J_{G_{1i}, G_{2i}}).
\]
Since the  summands  $J_{G_{1i}, G_{2i}}$ in $Q$ are ideals in pairwise  different sets of variables, it follows that $P=(\{x_{ij} : (i,j) \in W \}, \sum_{i=1}^rP_i)$, where  each $P_i$ is a minimal prime ideal of  $J_{G_{1i}, G_{2i}}$.  None of the $P_i$ contains a variable. It follows therefore from Lemma~\ref{novariables} that $P_i=I_2 ((x_{kl})_{ k \in V(G_{1i})\atop l \in V(G_{2i})}$ for $i=1,\ldots,r$, as desired.

(b) follows from Proposition~\ref{contained}.
\end{proof}

Among the minimal prime ideals of $J_{G_1,G_2}$ are those which are only determined by the data of $G_1$, respectively those by $G_2$. To explain this, let $G$ be a finite simple graph on the vertex set $[n]$. A subset $S\subset G$ is said to have the {\em cut point property} if each $i\in S$ is a cut point of the  graph $G_{[n]\setminus S}$. In other words, $S$ has the cut point property, if for all $i\in S$,  the number of connected components of $G_{([n]\setminus S)\union \{i\}}$ is smaller than that of $G_{[n]\setminus S}$.

\begin{Proposition} \label{special}
Let $S_1\subset V(G_1)=[m]$ and $S_2\subset V(G_2)=[n]$ be subsets with the cut point property. Then $W_1=S_1\times [n]$ and $W_2=[m]\times S_2$ are admissible sets and $P_{W_1}$ and $P_{W_2}$ are minimal prime ideals of $J_{G_1,G_2}$.
\end{Proposition}

\begin{proof}
By symmetry it is enough to show that  $W_1$ is admissible and that $P_{W_1}$ is a minimal prime ideal. The set $W_1$ being admissible is obvious. Now let $V\subset W_1$ be an  admissible set which is a proper subset of $W_1$. Then $V=T\times [n]$ where $T\subset S_1$ is a proper subset of $S$. Since $S$ has the cut point property it follows that $(G_1)_{[n]\setminus T}$ has less connected components than $(G_1)_{[n]\setminus S}$. Let $G$ be a connected component of $(G_1)_{[n]\setminus T}$ which is not a connected component of $(G_1)_{[n]\setminus S}$. Then there exist two vertices  $i,j \in V(G)$ which are not connected in $(G_1)_{[n]\setminus S}$. Therefore, for any $f \in E(G_2)$ the set $\{i,j\} \times f$ is contained in $\widehat{V}^c \setminus \widehat{W}^c$ and does not have any edge in $W$. Thus it follows from Theorem~\ref{minprimes}(b) that $P_{W_1}$ is a minimal prime ideal of $J_{G_1,G_2}$.
\end{proof}

\section{The case $3\times n$}

In this section we aim at describing explicitly the minimal prime ideals of $J_{G_1,G_2}$   in the case that $|V(G_1)|=3$.

 Let $G_1$ be a connected graph on vertex set $[3]$ and $G_2$ be a connected graph on vertex set $[n]$. The graph $G_1$ is either a path graph or a complete graph. In the case of a complete graph the minimal prime ideals are known by \cite{RA}. Here we want to analyze the case when $G_1$ is a line graph with edges $\{1,2\}$ and $\{2,3\}$.

 Let $T$ be any subset of $[n]$, and let $C_1, \ldots C_r$ be the connected components of $(G_2)_{[n]\setminus T}$. Furthermore, let $B$ be a subset of $[r]$. We set
 \begin{eqnarray}
\label{set}
W_{T,B}= ([3] \times T) \cup \bigcup_{j \in B} (\{2 \} \times V(C_j)).
\end{eqnarray}

Note that $W_{T,B}$ is an admissible set with respect to $(G_1, G_2)$. We are going to prove that any admissible set $W$ for which $P_W$ is a minimal prime ideal of $J_{G_1,G_2}$, is of the form $W_{T,B}$, where $T$ and $B$ satisfy some extra conditions.

We first show

\begin{Lemma}
\label{lemma1}
Let $P_W$ be a minimal prime ideal of $J_{G_1, G_2}$. Suppose there exists some $(i,s) \in W$ with $i \in \{1,3\}$ and $s \in [n]$. Then $[3] \times s \subset W$.
\end{Lemma}

\begin{proof}
%Let $A=\{(i,r) \in W \:\; i \in \{1,3\}, [3] \times r \not \subset W\}$, and $W'= W \setminus A$.
Let
\[
W'=\{(2,r) \; \: \; (2,r) \in W\} \cup \bigcup_{[3] \times \{r\} \subset W} [3] \times \{r\}.
\]
We first show that $W'$ is an admissible set with respect to $(G_1, G_2)$. Let $(i,r) \in e \times f \cap W'$ for some $e \in E(G_1)$ and for some $f \in E(G_2)$. If $[3] \times \{r\} \subset W$, then $[3] \times \{r\} \subset W'$, in particular, $e \times \{r\} \subset W'$. Otherwise, we may assume that $i=2$ and $[3] \times \{r\} \not \subset W$. Then $\{2\} \times f \in W$ because $W$ is admissible, and hence $\{2\} \times f \in W'$. Therefore, $W'$ is admissible.
%Since $(k,l) \in W'$, we either have $[3] \times l  \in W'$, or $k=2$ and either $(1,l)$  or $(3,l) \in W$. In the first case $e \times \{l\} \in W'$. In the second case, admissibility of $W$ implies that $\{2\} \times f \in W$, which gives $\{2\} \times f \in W'$. %To see this, observe that both $\{i,s\} \times f$ and $\{s,j\} \times f$ have nonempty intersection with $W$ and either $[3] \times l \nsubseteq W$.

Assume that $W' \neq W$. We claim that in this case $P_{W'}$ is properly contained in $P_W$, contradicting the assumption that $P_{W}$ is minimal prime ideal. Indeed, $W'$ is proper subset of $W$. Let $e \times f \in \widehat{W}'^c \setminus \widehat{W}^c$. We may assume that $e=\{1,2\}$. Then $\{1\} \times f \subset W$ because $\{2\} \times f \not \subset W$ and $W$ is admissible.
\end{proof}

In the following, we will have to refer to the following operations on graphs. Let $G$ be a graph and $H$ be a subgraph of $G$. Then $G \setminus \{i\}$ denotes the subgraph of $G$ which is obtained by removing the vertex $i$ along with all the edges incident to $i$, and $H \cup \{i\}$ denotes the subgraph of $G$ which is obtained by adding to $H$ the vertex $i$ and all the edges of $G$ which connect $i$ with $H$.

\begin{Lemma} \label{lemma2}
Let $P_W$ be a minimal prime ideal of $J_{G_1,G_2}$, and let $T=\{a \in [n] : [3] \times \{a\} \in W\}$. Then $T$ has the cut point property.
\end{Lemma}

\begin{proof}
Assume that $T$ does not have the cut point property. Then there exists an element $a \in T$ such that $(G_2)_{[n]\setminus T}$ has same number of connected components as $(G_2)_{([n]/T) \cup \{a\}}$. This implies that  there exists a unique connected  component $D$ of  $(G_2)_{[n]/T \cup \{a\}}$ which contains $a$ and such that $C=D \setminus \{a\}$ is connected. %Here $D \setminus \{a\}$ denotes the subgraph of $D$ which is obtained by removing the vertex $a$ along with all the edges incident to $a$.

We set $W' = W \setminus ( [3] \times \{a\} )$ if $W \cap ([3] \times V(C)) = \emptyset$, otherwise we set $W' = W \setminus \{(1,a), (3,a)\}$. By using Lemma~\ref{lemma1}, it follows that $W'$ is of the form $W_{T,B}$ as described in $(\ref{set})$. Therefore, $W'$ is admissible.

We claim that $P_{W'} \subsetneq P_W$. By using Proposition~\ref{contained}, it is enough to show that for all $e \times f \subset \widehat{W}'^c \setminus \widehat{W}^c$ an edge of $e \times f$ is contained in $W$. In the case when $W'= W \setminus ([3] \times \{a\})$, any $e \times f \subset \widehat{W}'^c \setminus \widehat{W}^c$ has an edge in $[3] \times \{a\}$. In the case when $W'= W \setminus \{(1,a),(3,a)\}$ we have $\widehat{W}'^c = \widehat{W}^c$. Therefore, our claim holds and we obtain a contradiction to the minimality of $P_W$.%Since there does not exist any $e \times f$ contained in $W'^c \setminus W^c$, our claim holds and we obtain a contradiction to the minimality of $P_W$.
\end{proof}

Now we are ready to describe the minimal prime ideals of $J_{G_1,G_2}$.

\begin{Theorem}\label{minprime3timesn}
 Let $W$ be an admissible set with respect to  $(G_1,G_2)$. Then the following conditions are equivalent:
\begin{enumerate}
\item[{\em (a)}] $P_W$ is a minimal prime ideal of $J_{G_1,G_2}$.
\item[{\em (b)}] $W=W_{T,B}$, where $T$ and $B$ satisfy the following conditions:

\; \item[{\em (i)}] $T$ has the cut point property with respect to $G_2$;
\; \item[{\em (ii)}] Let $C_1,\ldots,C_r$ be the connected components of  $(G_2)_{[n]\setminus T}$. Then
\begin{enumerate}
\item[{\em ($\alpha$)}]  $|V(C_j)| \geq 2$ for $j \in  B$;

\item[{\em ($\beta$)}]  for all $k,l \in B$ with $k\neq l$,  $(C_k  \cup C_l) \cup \{a\}$ is disconnected  for all $a \in T$.
\end{enumerate}
\end{enumerate}
\end{Theorem}

\begin{proof}
(a) $\Rightarrow$ (b): We know from Lemma~\ref{lemma1} and Lemma~\ref{lemma2} that $W=W_{T,B}$ where $T$ has the cut point property with respect to $G_2$. Suppose $|V(C_j)| =1$ for some $j \in B$, then $C_j = \{a\}$ for some $a\in V(G_2)$, and $W' = W\setminus \{(2,a)\}$ is admissible  with $P_{W'} \subsetneq P_W$, a contradiction. This proves condition $(\alpha)$.

Suppose there exist $a \in T$ such that $(C_k  \cup C_l) \cup \{a\}$ is connected in $G_2$ for some $k,l \in B$ with $k\neq l$. Let $W' = W\setminus \{(1,a) , (3,a)\}$. Then $W'$ is admissible and $P_{W'} \subsetneq P_W$, a contradiction. This proves $(\beta)$.

(b) $\Rightarrow$ (a): Assume that $P_{W_{T,B}}$ is not a minimal prime ideal of $J_{G_1,G_2}$. Then there exist a minimal prime ideal $Q \subsetneq P_{W_{T,B}}$ of $J_{G_1,G_2}$. By the implication (a) \implies (b), which is already shown, it follows that $Q=P_{W_{T',B'}}$ with $T' \subset T$ and $B' \subset B$.
%exist an admissible set $W_{T',B'}$ with $T' \subset T$ and $B' \subset B$ such that $P_{W_{T',B'}} \subsetneq P_{W_{T,B}}$.
 Suppose $T' \subsetneq T$. Since $T$ has the cut point property, there exist two connected components $C_k, C_l$ of $(G_2)_{[n]\setminus T}$ and $a \in T \setminus T'$ such that $(C_k \cup C_l) \cup \{a\}$ is connected. Let $i \in V(C_k)$ and $j \in V(C_l)$ and $e \in E(G_1)$. Then $ e \times \{i,j\}$ is contained in $\widehat{W}_{T',B'}^c \setminus \widehat{W}^c_{T,B}$. It is clear that the edges $e \times \{i\}$, $e \times \{j\}$ and $\{1\}\times\{i,j\}$, if $e=\{1,2\}$, respectively,  $\{3\}\times\{i,j\}$, if $e=\{2,3\}$,   are not contained in $W_{T,B}$. But also the edge $\{2\} \times \{i,j\}$ is not contained in $W_{T,B}$ because of condition ($\beta$). Therefore, it follows from Proposition~\ref{contained} that  $P_{W_{T',B'}} \nsubseteq P_{W_{T,B}}$, a contradiction. Hence we have $T'=T$. Therefore, we must have $B' \subsetneq B$.
 % and let $G$ be a connected component of $(G_2)_{[n]\setminus T'}$ which is not a connected component of $(G_2)_{[n]\setminus T}$. Then there exist two vertices  $i,j \in V(G)$ which are not connected in $(G_2)_{[n]\setminus T}$. Therefore, for any $e \in E(G_1)$ the set $ e \times \{i,j\}$ is contained in $W_{T',B'}^c \setminus W^c_{T,B}$ and does not have any edge in $W_{T,B}$. Therefore, it follows from Proposition~\ref{contained} that  $P_{W_{T',B'}} \nsubseteq P_{W_{T,B}}$, a contradiction. Hence we have $T'=T$, but then  $B' \subsetneq B$.
Then there exist $k \in B \setminus B'$ such that $([3] \times V(C_k)) \cap W_{T,B'} = \emptyset$. By condition ($\alpha$) there exist $i,j \in V(C_k)$ with $i \neq j$.  Therefore $\{1,3\} \times  \{i,j\}$ is contained in $\widehat{W}^c_{T,B'} \setminus \widehat{W}^c_{T,B} $ and has no edge in $W_{T,B}$. It again gives a contradiction to our assumption that $P_{W_{T,B'}} \subsetneq P_{W_{T,B}}$.
\end{proof}
%Then there exist an element $a \in T \setminus T'$ and $k,l \in B$ such that $C_k \cup \{a\} \cup C_l$ is a connected in $G_{2_{[n] \setminus T'}}$, and $\{i,j\} \times f $ is contained in  $W_{T',B'}^c \setminus W^c_{T,B}$ for some $i \in V(C_k)$, $j \in V(C_l)$ and $f \in E(G_1)$. Since $a \in T$, and $C_k$ and $C_l$ are not connected in $W_{T,B}$, no edge of $\{i,j\} \times f $ is contained in $W_{T,B}$, a contradiction to our assumption that $P_{W_{T',B'}} \subsetneq P_{W_{T,B}}$. Therefore, we have $T'=T$.

In \cite[Theorem 3.1]{HS}, Ho\c{s}ten and Shapiro describe the minimal prime ideals of the ideal of adjacent 2-minors of a $3 \times n$ matrix. In our language, these are the minimal prime ideals of $J_{G_1, G_2}$ where $G_1$ and $G_2$ are line graphs with $|V(G_1)|=3$ and $|V(G_2)|=n$. By using the fact that in this particular case the subsets $T=\{a_1, \ldots, a_r\}$ of $V(G_2)=[n]$ with the cut point property are of the form $1 < a_1$, $a_r < n$ and $a_i < a_{i+1} -1$ for $i= 1 \ldots r-1$, we obtained the result of Ho\c{s}ten and Shapiro as a special case of Theorem~\ref{minprime3timesn}.

In Figure~\ref{minimalp}, we display the admissible sets, marked by fat dots, attached with the minimal prime ideals of $J_{G_1, G_2}$ where $G_1$ is a line graph of length 2 and $G_2$ is graph on vertex set $[5]$ with edge set $\{\{1,2\}, \{2,3\}, \{3,4\}, \{1,4\}, \{4,5\}\}$. %The graph $G_2$ has 4 sets with cut point property, namely, $\{2\}, \{4\}, \{2,4\}, \{1,3\}$.

\begin{figure}[hbt]
\begin{center}
\psset{unit=0.7cm}
\begin{pspicture}(4.5,-3)(4.5,2)

\rput(2,0){
\pspolygon[style=fyp,fillcolor=light](-5,0)(-5,1)(-4,1)(-4,0)
\pspolygon[style=fyp,fillcolor=light](-5,1)(-5,2)(-4,2)(-4,1)
\pspolygon[style=fyp,fillcolor=light](-4,0)(-4,1)(-3,1)(-3,0)
\pspolygon[style=fyp,fillcolor=light](-4,1)(-4,2)(-3,2)(-3,1)
\pspolygon[style=fyp,fillcolor=light](-3,0)(-3,1)(-2,1)(-2,0)
\pspolygon[style=fyp,fillcolor=light](-3,1)(-3,2)(-2,2)(-2,1)
\pspolygon[style=fyp,fillcolor=light](-2,0)(-2,1)(-1,1)(-1,0)
\pspolygon[style=fyp,fillcolor=light](-2,1)(-2,2)(-1,2)(-1,1)
\rput(-5,1){$\bullet$}
\rput(-4,1){$\bullet$}
\rput(-3,1){$\bullet$}
\rput(-2,1){$\bullet$}
\rput(-1,1){$\bullet$}
}

\rput(7,0){
\pspolygon[style=fyp,fillcolor=light](-5,0)(-5,1)(-4,1)(-4,0)
\pspolygon[style=fyp,fillcolor=light](-5,1)(-5,2)(-4,2)(-4,1)
\pspolygon[style=fyp,fillcolor=light](-4,0)(-4,1)(-3,1)(-3,0)
\pspolygon[style=fyp,fillcolor=light](-4,1)(-4,2)(-3,2)(-3,1)
\pspolygon[style=fyp,fillcolor=light](-3,0)(-3,1)(-2,1)(-2,0)
\pspolygon[style=fyp,fillcolor=light](-3,1)(-3,2)(-2,2)(-2,1)
\pspolygon[style=fyp,fillcolor=light](-2,0)(-2,1)(-1,1)(-1,0)
\pspolygon[style=fyp,fillcolor=light](-2,1)(-2,2)(-1,2)(-1,1)
\rput(-4,0){$\bullet$}
\rput(-4,1){$\bullet$}
\rput(-4,2){$\bullet$}
\rput(-2,0){$\bullet$}
\rput(-2,1){$\bullet$}
\rput(-2,2){$\bullet$}
}

\rput(12,0){
\pspolygon[style=fyp,fillcolor=light](-5,0)(-5,1)(-4,1)(-4,0)
\pspolygon[style=fyp,fillcolor=light](-5,1)(-5,2)(-4,2)(-4,1)
\pspolygon[style=fyp,fillcolor=light](-4,0)(-4,1)(-3,1)(-3,0)
\pspolygon[style=fyp,fillcolor=light](-4,1)(-4,2)(-3,2)(-3,1)
\pspolygon[style=fyp,fillcolor=light](-3,0)(-3,1)(-2,1)(-2,0)
\pspolygon[style=fyp,fillcolor=light](-3,1)(-3,2)(-2,2)(-2,1)
\pspolygon[style=fyp,fillcolor=light](-2,0)(-2,1)(-1,1)(-1,0)
\pspolygon[style=fyp,fillcolor=light](-2,1)(-2,2)(-1,2)(-1,1)
\rput(-2,0){$\bullet$}
\rput(-2,1){$\bullet$}
\rput(-2,2){$\bullet$}
}

\rput(2,-3){
\pspolygon[style=fyp,fillcolor=light](-5,0)(-5,1)(-4,1)(-4,0)
\pspolygon[style=fyp,fillcolor=light](-5,1)(-5,2)(-4,2)(-4,1)
\pspolygon[style=fyp,fillcolor=light](-4,0)(-4,1)(-3,1)(-3,0)
\pspolygon[style=fyp,fillcolor=light](-4,1)(-4,2)(-3,2)(-3,1)
\pspolygon[style=fyp,fillcolor=light](-3,0)(-3,1)(-2,1)(-2,0)
\pspolygon[style=fyp,fillcolor=light](-3,1)(-3,2)(-2,2)(-2,1)
\pspolygon[style=fyp,fillcolor=light](-2,0)(-2,1)(-1,1)(-1,0)
\pspolygon[style=fyp,fillcolor=light](-2,1)(-2,2)(-1,2)(-1,1)
\rput(-2,0){$\bullet$}
\rput(-2,1){$\bullet$}
\rput(-2,2){$\bullet$}
\rput(-5,1){$\bullet$}
\rput(-4,1){$\bullet$}
\rput(-3,1){$\bullet$}
\rput(-2,1){$\bullet$}
}

\rput(7,-3){
\pspolygon[style=fyp,fillcolor=light](-5,0)(-5,1)(-4,1)(-4,0)
\pspolygon[style=fyp,fillcolor=light](-5,1)(-5,2)(-4,2)(-4,1)
\pspolygon[style=fyp,fillcolor=light](-4,0)(-4,1)(-3,1)(-3,0)
\pspolygon[style=fyp,fillcolor=light](-4,1)(-4,2)(-3,2)(-3,1)
\pspolygon[style=fyp,fillcolor=light](-3,0)(-3,1)(-2,1)(-2,0)
\pspolygon[style=fyp,fillcolor=light](-3,1)(-3,2)(-2,2)(-2,1)
\pspolygon[style=fyp,fillcolor=light](-2,0)(-2,1)(-1,1)(-1,0)
\pspolygon[style=fyp,fillcolor=light](-2,1)(-2,2)(-1,2)(-1,1)
\rput(-5,0){$\bullet$}
\rput(-5,1){$\bullet$}
\rput(-5,2){$\bullet$}
\rput(-3,0){$\bullet$}
\rput(-3,1){$\bullet$}
\rput(-3,2){$\bullet$}
}

\rput(12,-3){
\pspolygon[style=fyp,fillcolor=light](-5,0)(-5,1)(-4,1)(-4,0)
\pspolygon[style=fyp,fillcolor=light](-5,1)(-5,2)(-4,2)(-4,1)
\pspolygon[style=fyp,fillcolor=light](-4,0)(-4,1)(-3,1)(-3,0)
\pspolygon[style=fyp,fillcolor=light](-4,1)(-4,2)(-3,2)(-3,1)
\pspolygon[style=fyp,fillcolor=light](-3,0)(-3,1)(-2,1)(-2,0)
\pspolygon[style=fyp,fillcolor=light](-3,1)(-3,2)(-2,2)(-2,1)
\pspolygon[style=fyp,fillcolor=light](-2,0)(-2,1)(-1,1)(-1,0)
\pspolygon[style=fyp,fillcolor=light](-2,1)(-2,2)(-1,2)(-1,1)
\rput(-5,0){$\bullet$}
\rput(-5,1){$\bullet$}
\rput(-5,2){$\bullet$}
\rput(-3,0){$\bullet$}
\rput(-3,1){$\bullet$}
\rput(-3,2){$\bullet$}
\rput(-2,1){$\bullet$}
\rput(-1,1){$\bullet$}
}

\end{pspicture}
\end{center}
\caption{}\label{minimalp}
\end{figure}

\section{Unmixed binomial ideals of pairs of graphs}

In this section we classify all pairs of graphs $(G_1,G_2)$ such that $J_{G_1,G_2}$ is unmixed, and those for which $J_{G_1,G_2}$ is Cohen--Macaulay,  under the additional assumption that the graphs are closed.

\begin{Proposition}
\label{unmixed}
Let $n\geq m\geq 3$ be integers and let $G_1$ and $G_2$ be connected simple graphs with $V(G_1)=[m]$ and $V(G_2)=[n]$. Then the binomial edge ideal $J_{G_1,G_2}$ is unmixed if and only if $G_1$ is complete and  for all subsets $T\subset [n]$ with the cut point property for $G_2$ one has
\begin{eqnarray}
\label{numerical}
(c(T)-1)(m-1)=|T|.
\end{eqnarray}
\end{Proposition}

\begin{proof}
Assume that $J_{G_1,G_2}$ is unmixed and let us suppose  that $G_1$ is not complete. Since $I_2(X)$ is one of the minimal primes of $J_{G_1,G_2}$ with height $(m-1)(n-1)$, all the other minimal prime ideals of $J_{G_1,G_2}$ must have the same height. By  Proposition~\ref{special}, any prime ideal  $P_W,$ where  $W =S \times [n]$ and $\emptyset\neq S \subset [m]$ has the cut point property for $G_1,$ is a minimal prime of $J_{G_1,G_2}$. Let $G'_1, \ldots, G'_{c(S)}$ be the connected components of $(G_1)_{([m] \setminus S)}$, and  $g_i = |V(G'_i)|$ for $i=1,\ldots,c(s)$. Then $\sum_{i=1}^{c(S)} g_i= m-|S|$ and
\[
\height P_W = n|S| + \sum_{i=1}^{c(S)} (g_i - 1) (n-1)=n|S| +(m-|S| - c(S))(n-1).
\]
Hence, since $J_{G_1,G_2}$ is unmixed, we get $(c(S)-1)(n-1)=|S|$. Moreover, we have  $|S|\geq n-1\geq m-1.$ But it is obvious that no
$(m-1)$-subset of $[m]$ has the cut point property for $G_1,$ therefore, $G_1$ must be complete. By using  arguments as in the first part of the proof for the graph $G_2$, one gets  condition~(\ref{numerical}).

For the converse, we use a result of \cite{RA} which says that if $G_1$ is complete, then the minimal prime ideals of $J_{G_1,G_2}$ are exactly the prime ideals $P_W$ with $W=[m]\times T$ where  $T\subset [n]$ is  a set with the cut point property for $G_2$.   The numerical condition~(\ref{numerical}) shows that  these prime ideals have all the  same height, hence $J$ is unmixed.
\end{proof}

The above proposition and Theorem~\ref{radical} show, in particular, that an unmixed ideal associated with a pair of graphs is  radical. It is very easy to see that the converse is not true. For instance, one may take $G_1$ the complete graph on $[3]$ and $G_2$ the line graph with the edges $\{1,2\},\{2,3\}$. The ideal $J_{G_1,G_2}$ is radical, by Theorem~\ref{radical}, and it is not unmixed, since its minimal prime ideals have different heights.

Proposition~\ref{unmixed} shows also that $J_{G_1,G_2}$ is not unmixed for any connected graph $G_2$ which has a nonempty set $T$ with the cut point property such that $m-1$ does not divide $|T|.$ In particular, if $G_2$ is a tree, the ideal $J_{G_1,G_2}$ is not unmixed, since we may find subsets $T\subset [n]$ with the cut point property of cardinality $1$. In the next proposition we discuss the unmixedness for the case when $G_2$ is a cycle.

\begin{Proposition}\label{unmixed cycle}
Let $n\geq m\geq 3$ and let $G_1$ be the complete graph on $[m]$ and $G_2$ the cycle on the set $[n].$ Then
$J_{G_1,G_2}$ is unmixed if and only if $m=n=3$ or $n=4,m=3$ or $n=5, m=3.$
	%\item [(ii)] $J_{G_1,G_2}$ is Cohen-Macaulay if and only if $m=n=3.$
\end{Proposition}

\begin{proof}
(i) By Proposition~\ref{unmixed}, $J_{G_1,G_2}$ is unmixed if and only if, for every subset $T\subset [n]$ which has the cut point
property for $G_2$, we have
\begin{equation}\label{numeric}
(c(T)-1)(m-1)=|T|.
\end{equation}
If $n\geq 6,$ there exists subsets $T$ of $[n]$ with the cut point property such that $c(T)=|T|=3.$ Hence, we get $2(m-1)=3,$ which is impossible.
Therefore, for unmixedness  we must restrict to $n=3,4$ or $5.$ If $m=n=3$ the claims are obvious since
$J_{G_1,G_2}$ is the ideal of all $2$-minors of the matrix $X$.

Let $n=4$ and assume that $G_2$ has the edges $\{1,2\},\{2,3\},\{3,4\}$ and $\{4,1\}.$ Then the sets with the cut point property for $G_2$ are
$\emptyset, \{1,3\}$ and $\{2,4\}.$ By using (\ref{numeric}) for a set $T$ with two elements, we get $m-1=2,$ hence $m=3.$ In this case
 all the minimal prime ideals of $J_{G_1,G_2}$ have the same height equal to $6$.

Let $n=5.$ In this case we see again that the nonempty subsets of $[5]$ with the cut point property for $G_2$ are of cardinality $2,$ and, as in the
case $n=4$, we obtain $m=3.$
\end{proof}

\begin{Remark}
{\em By using the computer, one easily sees that, in the hypotheses of the above proposition, $J_{G_1,G_2}$ is Cohen-Macaulay if and only if $m=n=3.$}
\end{Remark}

Closed graphs form an  interesting class of graphs $G_2$ for which one may discuss the unmixedness property. We recall that the collection of cliques of a graph $G$ forms a simplicial complex, called the {\em clique complex} of $G.$ We denote it $\Delta(G)$. In  \cite[Theorem 2.2]{EHH} it is shown that  a graph
$G$ on the vertex set $[n]$ is closed if and only if there exists a labeling of $G$ such that all the facets of $\Delta(G)$ are intervals $[a,b]\subset [n]$. Moreover, if one labels the facets $F_1,\ldots,F_r$ of $\Delta(G)$ such that $\min(F_1)<\min(F_2)<\cdots <\min(F_r),$
then $F_1,\ldots,F_r$ is a leaf order of $\Delta(G).$

\begin{Theorem}\label{Osaka}
Let $n\geq m\geq 3$ be integers, let $G_1$ be the complete graph on $[m]$, and $G_2$ a connected closed graph on $[n]$. The following conditions are equivalent:
\begin{itemize}
	\item [(i)] $J_{G_1,G_2}$ is unmixed.
	\item [(ii)] There exists a leaf order $F_1,\ldots,F_r$ of the facets of $\Delta(G_2)$ such that, for $1\leq i\leq r$ $F_i=[a_i,b_i],$  where $a_i,b_i$ are positive integers with  $a_i< a_{i+1}<b_i< b_{i+1}$ and $b_i-a_{i+1}=m-2$ for $1\leq i\leq r-1.$
\end{itemize} Moreover, in the above conditions,  \[\depth(S/J_{G_1,G_2})= n-(r-2)m+2r-3,\] where $r$ is the number of the facets of the clique complex $\Delta(G_2).$ Consequently, $S/J_{G_1,G_2}$ is  Cohen-Macaulay if and only if $G_2$ is a complete graph.
\end{Theorem}

\begin{proof}

By Theorem 2.2 in \cite{EHH}, the clique complex $\Delta(G_2)$ has the facets $F_1,\ldots,F_r$ where each facet is an interval, that is,  $F_i=[a_i,b_i]$
and $1=a_1<a_2<\cdots < a_r\leq b_r=n.$ Since $G_2$ is connected, it follows that $a_{i+1}\leq b_i$ for all $i.$

For (i) $\Rightarrow$ (ii) we proceed by induction on $r.$
Let $T=[a_r,b_{r-1}].$ Then $T$ has the cut point property and $c(T)=2,$ thus, by (\ref{numeric}), we get
\[
b_{r-1}-a_r+1=m-1.
\]
Let $G_2^\prime$ be the graph whose clique complex $\Delta(G_2^\prime)$ has the facets $F_1,\ldots,F_{r-1}$ and let $P_{W^\prime}$ be a minimal prime of $J_{G_1,G_2^\prime}$ where $W^\prime=[m]\times T^\prime,$ with $T^\prime\subset V(G_2^\prime)$ a set with the cut point property for $G_2^\prime.$  Then $b_{r-1}\not\in T^\prime,$ thus, $c_{G_2^\prime}(T^\prime)=c_{G_2}(T^\prime).$ It follows that $T^\prime$ has the cut point property for  $G_2$ as well. Therefore, $T^\prime$ satisfies condition~(\ref{numeric}), so we may apply induction.

For (ii) $\Rightarrow$ (i) and for the formula of the depth we apply again induction on $r.$ For $r=1$ there is nothing to prove since $J_{G_1,G_2}=I_2(X).$ In particular, $S/J_{G_1,G_2}$ is  Cohen-Macaulay  of depth $m+n-1.$

Let now $r>1$ and $G_2$ a closed graph whose clique complex has $r$ facets, $F_1,\ldots,F_r$. For each subset $T$ of $[n]$ with the cut point
property for $G_2$, we denote by $P_T(J)$ the minimal prime ideal of $J=J_{G_1,G_2}$ which corresponds to the admissible set $W=[m]\times T.$ Let
$T_0=[a_r,b_{r-1}]$ and set
\[
J^\prime= \bigcap\limits_{P_T(J)\in \Min(J)\atop T\not\supset T_0}P_T(J), \  J^{\prime\prime}= \bigcap\limits_{P_T(J)\in \Min(J)\atop T\supset
T_0}P_T(J),
\]
where $\Min(J)$ is the set of the minimal prime ideals of $J.$
Then $J=J^\prime \cap J^{\prime\prime}$, hence, in order to prove the unmixedness of $J$ we have to show that $J^\prime$ and $J^{\prime\prime}$ are unmixed of the same height equal to $(m-1)(n-1)$.

We note that $J^\prime=J_{G_1,G_2^\prime}$ where $G_2^\prime$ is obtained from $G_2$ by replacing the facets $F_{r-1}$ and $F_r$ of $\Delta(G_2)$ with the clique on
the set  $[a_{r-1}, n].$ Therefore, $G_2^\prime $ has $r-1$ cliques and $J^\prime$ is unmixed, by induction.  In addition, again by induction, we get
\[
\depth(S/J^\prime)=n-(r-3)m+2r-5.
\]
On the other hand, $J^{\prime\prime}=(\{x_{ij}: (i,j)\in [m]\times T_0\})+J_{G_1,G_2^{\prime\prime}}$ where $G_2^{\prime\prime}$ is the restriction of
$G_2$ to the vertex set $[n]\setminus T_0.$ It follows that $G_2^{\prime\prime}$ has two connected components. let us denote them $H_1$ and $H_2$,
where $H_1$ is given by $r-1$ cliques on the vertex set $[a_r-1]$ and $H_2$ is the clique on the vertex set $[b_{r-1}+1,n]$. Therefore, by the inductive hypothesis, it follows that $J_{G_1,H_1}$ is unmixed of height $(m-1)(a_r-2).$ This implies that every minimal prime of  $J^{\prime\prime}$ has height equal to $m|T_0|+(m-1)(a_r-2)+(m-1)(n-b_{r-1}-1)=(m-1)(n-1),$ thus $J^{\prime\prime}$ is also unmixed. This ends the proof of unmixedness of $J_{G_1,G_2}.$

In order to finish the proof of depth's formula, we use the following exact sequence:
\begin{equation}\label{seqdepth}
0\to \frac{S}{J}\to \frac{S}{J^\prime}\oplus\frac{S}{J^{\prime\prime}}\to \frac{S}{J^\prime+ J^{\prime\prime}}\to 0.
\end{equation}
 It is clear from the  decomposition of $J^{\prime\prime}$ that
\begin{equation}\label{tensor}
\frac{S}{J^{\prime\prime}}\cong \frac{S_1}{J_{G_1,H_1}}\otimes_K \frac{S_2}{J_{G_1,H_2}}
\end{equation}
where $S_1$ is the polynomial ring in the variables $x_{ij}, (i,j)\in [m]\times [a_r-1]$ and $S_2$ is the polynomial ring in the variables
$x_{ij}, (i,j)\in [m]\times [b_{r-1}+1,n].$ Since $J_{G_1,H_1}$ is unmixed and $H_1$ has $r-1$ cliques, by induction, it follows that
$\depth(S_1/J_{G_1,H_1})=a_r-1-(r-3)m+2r-5=a_r-(r-3)m+2r-6.$ Since $H_2$ is a clique, we get $\depth(S_2/J_{G_1,H_2})=n-b_{r-1}+m-1.$ Consequently, by
(\ref{tensor}), we obtain
\[
\depth(S/J^{\prime\prime})=n-(r-3)m+2r-5.
\]
Therefore,
\begin{equation}\label{depthmiddle}
\depth(S/J^\prime\oplus S/J^{\prime\prime})=n-(r-3)m+2r-5.
\end{equation}

Now we observe that $J^\prime+J^{\prime\prime}=J^\prime+(\{x_{ij}: (i,j)\in [m]\times T_0\})+J_{G_1,G_2^{\prime\prime}}=
J^\prime+ (\{x_{ij}: (i,j)\in [m]\times T_0\}$ since $J_{G_1,G_2^{\prime\prime}}$ is obviously contained in $J^\prime$. But this shows that
$S/(J^\prime+J^{\prime\prime})$ is nothing else than $S/J_{G_1,H}$ where $H$ is the graph obtained from $G_2^\prime$ by replacing its last clique on the vertex set $[a_{r-1},n]$ by the clique on the set $[a_{r-1},n]\setminus T_0.$  Therefore, $J_{G_1,H}$ is again unmixed and has $r-1$ cliques, so we may apply the inductive hypothesis. We then get
\begin{equation}
\depth(\frac{S}{J^\prime+J^{\prime\prime}})=\depth(\frac{S^\prime}{J_{G_1,H}})=n-|T_0|-(r-3)m+2r-5=n-(r-2)m+2r-4,
\end{equation}
where we denoted by $S^\prime$ the polynomial ring in the variables $x_{ij}$ with $(i,j)\in [m]\times ([n]\setminus T_0).$
Finally, by applying Depth Lemma in the sequence~(\ref{seqdepth}), we get
\[
\depth(S/J)=\depth(S/(J^\prime+J^{\prime\prime}))+1=n-(r-2)m+2r-3.
\]
The argument for the last claim in our theorem follows easily. If $G_2$ has $r$ cliques and  $J_{G_1,G_2}$ is Cohen-Macaulay, then the equality
$m+n-1=n-(r-2)m+2r-3$ must hold. Then we get $(r-1)m=2r-2$ which implies $m=2$ or $r=1.$ Hence, for $m\geq 3,$ $G_2$ must be complete.
\end{proof}

\section{A lower bound for the nilpotency index of $J_{G_1, G_2}$}

Let $I$ be an ideal in a Noetherian ring. Then there exists an integer $k$ such that $(\sqrt{I})^k \subset I$, where $\sqrt{I}$ denotes the radical of $I$. We call the  minimal number $k$ with this property the {\em index of nilpotency} of $I$ and denote it by $\nilpot(I)$. We have seen in Theorem~\ref{radical}, that $\nilpot(J_{G_1, G_2}) =1$ if and only if either $G_1$ or $G_2$ is complete. In this section we want to give a lower bound for $\nilpot(J_{G_1, G_2})$.

In the proof of the next result we shall need the following concept. Let $I$ be an ideal in a polynomial ring $S$ over a field, and let $X$ be a set of variables of $S$. We say that $I$ is {\em supported in $X$} if there exists a system of generators $f_1, \ldots, f_l$ of $I$ such that $X= \bigcup_{i=1}^{l} \supp(f_i)$, where for a polynomial $f$,  $\supp(f)$ denotes the set of variables which appear in $f$. If $I$ is supported in $X$, we call $X$ a {\em supporting set} of $I$.
\begin{Theorem}
\label{nilpot}
Let $T_1 \subset V(G_1)$ and $T_2 \subset V(G_2)$, and let $C_{11},\ldots, C_{1r}$ and $C_{21},\ldots, C_{2s}$ be those connected components of $(G_1)_{T_1}$ and $(G_2)_{T_2}$, respectively, which contain as an induced subgraph  a line graph of length at least 2. Then $\nilpot(J_{G_1, G_2}) \geq rs+1$.
\end{Theorem}

\begin{proof}
 In each $C_{ij}$ we choose a line graph $L_{ij}$ of length 2 which is an induced subgraph of  $C_{ij}$, and let $f_{ij} = f_{L_{1i} L_{2j}}$, as defined in (\ref{twolines}). Then, since $L_{1i}$ and $L_{2j}$ are also induced subgraphs of $G_1,$ respectively $G_2,$  it follows  that  $f_{ij} \notin J_{G_1, G_2}$, but $f^2_{ij} \in J_{G_1, G_2}$, as shown  in the proof of Theorem~\ref{radical}. Let $I$ be the ideal generated by the $f_{ij}$. Then  $I \subset \sqrt{J_{G_1, G_2}}$. We claim that $f=\prod_{i=1, \ldots , r \atop j=1, \ldots, s} f_{ij}$ does not belong to $J_{G_1,G_2}$. This then implies that $I^{rs} \nsubseteq J_{G_1,G_2}$, and we obtain the desired inequality for the nilpotency index for $J_{G_1,G_2}$.

In order to prove the claim, let $L= (\{x_{kl} \: \; (k,l) \in T_1 \times T_2\})$, and mark by `overline' reduction modulo $L$. Then $f= \bar{f}$ and
\[
\bar{J}_{G_1,G_2} = \sum_{i=1, \ldots , r \atop j=1, \ldots, s} J_{C_{1i}, C_{2j}} + J_0 ,
\]
where $J_0$ is the sum of the ideals of the form $J_{C,D}$ for the remaining connected components $C$ of $(G_1)_{T_1}$ and $D$ of $(G_2)_{T_2}$ which are different from the $C_{ij}$. Moreover, there exist supporting sets $X_{ij}$ for  $J_{C_{1i}, C_{2j}}$ and $X_0$ for $J_0$ (resulting from the generating  $2$-minors of these ideals)  such that $\supp(f_{ij}) \subset X_{ij}$ for all $i,j$, and such that all the supporting sets, including $X_0$, are pairwise disjoint.

Now suppose that $f\in J_{G_1, G_2}$. Then $f \in \bar{J}_{G_1, G_2}$ because $f = \bar{f}$. The next lemma however shows that   $f \notin \bar{J}_{G_1, G_2}$, a contradiction.  Thus $f \notin J_{G_1, G_2}$. This  proves the claim and the theorem.
\end{proof}

\begin{Lemma}
\label{useful}
Let $I_1, \ldots, I_r$ be ideals in a polynomial ring $S$ with supporting sets $X_1, \ldots, X_r$, and let $f_1, \ldots f_r$ be polynomials in $S$ such that $f_j \notin I_j$ for $j=1, \ldots, r$. Let $I=\sum_{i=1}^r I_j$ and $f=\prod_{i=1}^{r} f_i$, and suppose that
\begin{enumerate}
\item[{\em (i)}] $X_i \cap X_j = \emptyset$ for all $i \neq j$, and
\item[{\em (ii)}]  $\supp(f_i) \subset X_i$ for all $i$.
\end{enumerate}
Then $f \notin I$.
\end{Lemma}

\begin{proof}
Choose any monomial order $<$ on $S$. It follows from (i) that $\ini_<(I) = \sum_{j=1}^r \ini_<(I_j)$. Let $g_j$ be the remainder of $f_j$ with respect to a Gr\"obner basis of $I_j$. Since $f_j = g_j + h_j$ with $h_j \in I_j$, it follows that $f=\prod_{i=1}^r g_j +h$ where $h \in I$. Hence we see that $f \notin I$ if and only if $\prod_{i=1}^r g_j \notin I$. Thus we may replace the $f_j$ by the $g_j$, and hence may assume from the very beginning that $\ini_< (f_j) \notin \ini_<(I_j)$.

Suppose that $f \in I$. Then $\ini_<(f)\in\ini_<(I)$ and therefore $\prod_{i=1}^r\ini_<(f_i)\in \sum_{j=1}^r \ini_<(I_j)$. This implies that   for some $j$ there exists a monomial generator $u\in \ini_<(I_j)$  such that $u$ divides  $\prod_{i=1}^r\ini_<(f_i)$. Since $\supp(u)\subset X_j$, it follows from (i) and (ii) that $u$ divides $\ini_<(f_j)$. This is a contradiction, since $\ini_<(f_j)\notin \ini_<(I)$.
\end{proof}

We give a concrete example of Theorem~\ref{nilpot} in the form of the following corollary.

\begin{Corollary}
\label{notbest}
Let $J$ be the ideal of adjacent minors of an $m \times n$ matrix, and let  $k$ and $l$ be  integers such that $m=4k+p$ and $n=4l+q$ with $0\leq p,q < 4$. Then
\[
\nilpot(J) \geq (k + \lfloor \frac{p}{3} \rfloor) (l + \lfloor \frac{q}{3} \rfloor)+1 \approx \frac{mn}{16}.
\]
In particular, the index of nilpotency of the binomial edge of a pair of graphs can be arbitrarily big.
\end{Corollary}

\begin{proof}
We apply Theorem~\ref{nilpot}, in the case that $G_1$ is a line graph on $[m]$, and $G_2$ is a line graph on $[n]$. We choose $T_1=\{4a \:\; a \in [m], 4a \leq m\}$ and $T_2=\{4b \:\; b \in [n], 4b \leq n\}$. Then Theorem~\ref{nilpot} yields the desired conclusion.
\end{proof}

The bound given in Corollary~\ref{notbest} is definitely not the best possible. Calculations by computer show that  for $k=2,3,4$ the  index of nilpotency of the ideal of adjacent 2-minors of $3 \times 3k$ matrix is at least $k+1$. While our Corollary~\ref{notbest} only gives  $k$ as the lower bound for the index of nilpotency in these cases.

{}

\end{document}